\documentclass[12pt]{article} %submitted to   (2017- - )
\usepackage{mathrsfs}
\usepackage{amssymb}
\usepackage{amsmath}
\usepackage{pict2e}
\usepackage{amsmath}
\usepackage{amsfonts}
\usepackage{graphicx}
\usepackage[small]{caption}
\usepackage{subfigure}
\usepackage{cite}

\def\0{\emptyset}

\begin{document}
\newtheorem{claim}{Claim}[section]
\newtheorem{theorem}{Theorem}[section]
\newtheorem{corollary}[theorem]{Corollary}
\newtheorem{definition}[theorem]{Definition}
\newtheorem{conjecture}[theorem]{Conjecture}
\newtheorem{question}[theorem]{Question}
\newtheorem{lemma}[theorem]{Lemma}
\newtheorem{observation}[theorem]{Observation}
\newtheorem{proposition}[theorem]{Proposition}
\newenvironment{proof}{\noindent {\bf
Proof.}}{\rule{3mm}{3mm}\par\medskip}
\newcommand{\remark}{\medskip\par\noindent {\bf Remark.~~}}
\newcommand{\pp}{{\it p.}}
\newcommand{\de}{\em}

\title{\bf Structure connectivity and substructure connectivity of twisted hypercubes}

\author{Dong Li\footnote {Hubei Key Laboratory of Applied Mathematics, Faculty of Mathematics and Statistic, Hubei University, Wuhan 430062,
PR China}, ~Xiaolan Hu\footnote{School of Mathematics and Statistics $\&$ Hubei Key Laboratory of Mathematical Sciences, Central China Normal University,
Wuhan 430079, PR China} , ~Huiqing Liu$^*$
}

%\author{Huiqing Liu\footnote{Corresponding author. \textit{E-mail address:} hql\_2008@163.com(H.Q. Liu)}, Dong Li\\{\small  Faculty of Mathematics and Statistic, Hubei University, Wuhan 430062, PR China}}

\date{}
\maketitle \baselineskip 17.8pt

\begin{abstract}

Let $G$ be a graph and $T$ a certain connected subgraph of $G$. The $T$-structure connectivity $\kappa(G; T)$ (or resp., $T$-substructure connectivity $\kappa^{s}(G; T)$) of $G$ is the minimum number of a set of subgraphs $\mathcal{F}=\{T_{1}, T_{2}, \ldots, T_{m}\}$ (or resp., $\mathcal{F}=\{T^{'}_{1}, T^{'}_{2}, \ldots, T^{'}_{m}\}$) such that $T_{i}$ is isomorphic to $T$ (or resp., $T^{'}_{i}$ is a connected subgraph of $T$) for every $1\leq i \leq m$, and $\mathcal{F}$'s removal will disconnect $G$.
The twisted hypercube $H_{n}$ is a new variant of hypercubes with asymptotically optimal diameter introduced by X.D. Zhu. In this paper, we will determine both $\kappa(H_{n}; T)$ and $\kappa^{s}(H_{n}; T)$ for $T\in\{K_{1,r}, P_{k}\}$, respectively, where $3\leq r\leq 4$ and $1 \leq k \leq n$.

\vskip 0.1cm {\bf Keywords:}  Twisted hypercube; $T$-structure connectivity; $T$-substructure connectivity
\end{abstract}

\section{Introduction}

Interconnection networks play an important role in parallel and distributed systems. An interconnection network can be represented by an undirected graph $G =(V, E)$, where each vertex in $V$ corresponds to a processor, and every edge in $E$ corresponds to a communication link. The {\em neighborhood} $N_G(v)$ of a vertex $v$ in a graph $G =(V, E)$ is the set of vertices adjacent to $v$. For $S\subset V(G)$, the neighborhood $N_G(S)$ of $S$ in $G$ is defined as $N_G(S) =(\cup_{v\in S}N_G(v))-S$. We use $P_k=\langle v_1, v_2, \ldots, v_k\rangle$
and $C_k=\langle v_1, v_2, \ldots, v_k, v_1\rangle$
to denote a path and a cycle of order $k$, respectively. For $S\subset V(G)$, we use $G[S]$ to denote the subgraph of $G$ induced by $S$. For a subgraph $H$ of a graph $G$, we use $G-H$ to denote the subgraph of $G$ induced by $V(G)-V(H)$. For a set $\mathcal{F}=\{T_1, T_2, \ldots, T_m\}$, where each $T_i$ is isomorphic to a connected subgraph of $G$, we use $G-\mathcal{F}$ to denote the subgraph of $G$ induced by $V(G)-V(T_1)-V(T_2)-\cdots-V(T_m)$. For graph definition and notation not mentioned here we follow \cite{Bondy08}.

The {\em connectivity} $\kappa(G)$ of a graph $G$ is the minimum number of vertices whose removal leaves the remaining graph disconnected or trivial. The connectivity is one of the most important parameters to measure the reliability and fault tolerance of an interconnection network \cite{Chen03}, \cite{Esfahanian}, \cite{Fan03}, \cite{Harary}, \cite{Xu07}. The larger the connectivity is, more reliable the interconnection network is.
However, this parameter has a deficiency. That is, it tacitly assumes that all vertices adjacent to the same vertex of $G$ could fail at the same time, which is highly unlikely for large-scale systems. To compensate for this shortcoming, F\'{a}brega \cite{F} proposed the concept of $g$-extra connectivity. The {\em $g$-extra connectivity} of a graph $G$, denoted by $\kappa_g(G)$, is the minimum number of vertices of $G$ whose deletion disconnects $G$ and every remaining component has more than $g$ vertices. Some recent results on the $g$-extra connectivity of graphs see \cite{Chang13}, \cite{Hsieh12}, \cite{Yang09}, \cite{Zhu07}.

Instead of considering the effect of vertices becoming faulty, Lin {\em et al.} \cite{Lin16} considered the effect caused by some structures becoming faulty, and introduced the concept of structure connectivity and substructure connectivity of graphs.
Let $T$ be a connected subgraph of a graph $G$, and $\mathcal{F}$ a set of subgraphs of $G$ such that every element in $\mathcal{F}$ is isomorphic to $T$. Then $\mathcal{F}$ is called a {\em $T$-structure-cut} if $G-\mathcal{F}$ is disconnected. The {\em $T$-structure connectivity} $\kappa(G; T)$ of $G$ is defined as the cardinality of a minimum $T$-structure-cut of $G$. Similarly, let $\mathcal{F}$ be a set of subgraphs of $G$ such that every element in $\mathcal{F}$ is isomorphic to a connected subgraph of $T$. Then $\mathcal{F}$ is called a {\em $T$-substructure-cut} if $G-\mathcal{F}$ is disconnected. The {\em $T$-substructure connectivity} $\kappa^{s}(G; T)$ of $G$ is defined as the cardinality of a minimum $T$-substructure-cut of $G$. By definition, $\kappa^{s}(G; T)\leq \kappa(G; T)$. Note that $K_1$-structure connectivity and $K_1$-substructure connectivity are exactly the classical vertex connectivity.

Lin et al.\cite{Lin16} determined $\kappa(Q_n; T)$ and $\kappa^{s}(Q_n; T)$ for the hypercube $Q_n$ and $T\in\{K_{1,1}, K_{1,2},$ $ K_{1,3}, C_4\}$, respectively.
Sabir and Meng \cite{Sabir17} generalized their results and established $\kappa(Q_n; T)$ and $\kappa^{s}(Q_n; T)$ for $T\in\{P_k, C_{2k}, K_{1,4}\}$, where $3 \leq k \leq n$, they also determined $\kappa(FQ_n; T)$ and $\kappa^{s}(FQ_n; T)$ for  the folded hypercube $FQ_n$ and $T\in\{P_k, C_{2k}, K_{1,3}\}$, where $n \geq7$ and $2 \leq k \leq n$. Moreover, Mane \cite{Mane} determined $\kappa(Q_n; Q_m)$  with $m\leq n-2$ and established
the upper bound of $\kappa(Q_n; C_{2k})$ with  $2\leq k\leq 2^{n-1}$.
Recently, Lv et al.\cite{Lv18} explored $\kappa(Q_n^k; T)$ and $\kappa^{s}(Q_n^k; T)$ for the $k$-ary $n$-cube $Q_n^k$ and $T\in\{K_1,K_{1,1}, K_{1,2}, K_{1,3}\}$.

The interconnection network considered in this study is the twisted hypercube $H_n$, which is a hypercube-like network with asymptotically optimal diameter introduced by Zhu \cite{Zhu17}.  $H_n$ has many attractive properties,  such as low vertex degree, strong connectivity and super connectivity. Qi and Zhu \cite{Qi} considered the fault-diameter and wide-diameter of the twisted hypercubes. Liu {\em et al.} \cite{Liu} determined the $R^g$-vertex-connectivity and established the $g$-good neighbor conditional diagnosability of the twisted hypercubes under the PMC model and MM$^*$ model, respectively.

In this paper, we establish both $\kappa(H_n; T)$ and $\kappa^{s}(H_n; T)$ for $T\in\{K_{1,r}, P_k\}$, where $3\leq r\leq 4$ and $1\leq k \leq n$.
The rest of the paper is organized as follows. In Section 2, we introduce the definition of the $n$-dimensional twisted hypercube $H_n$ and provide preliminaries for our results. In Section 3, we determine $\kappa(H_n; K_{1,r})$ and $\kappa^{s}(H_n; K_{1,r})$ for $3\leq r\leq 4$. In Section 4, we determine $\kappa(H_n; P_k)$ and $\kappa^{s}(H_n; P_k)$ for $1 \leq k \leq n$. Our conclusions are given in Section 5.

\section{Preliminaries}

In this section, we first introduce the definition of the $n$-dimensional twisted hypercube $H_n$, then present some properties of $H_n$.

Denote by $Z_2^n$ the set of binary strings of length $n$. For $x,~y\in Z^n_2$, $x\oplus y$ denotes the sum of $x$ and
$y$ in the group $Z_2^n$, i.e., $(x\oplus y)_i = x_i + y_i (mod~2)$ (for $x\in Z_2^n$, $x_i$ denotes the $i$th bit of $x$).

If $x$ is a binary string of length $n_1$ and $y$ is a binary string of length $n_2$, then $xy$ is the concatenation of $x$ and $y$, which is a binary string of length $n_1 +n_2$. If $Z$ is a set of binary strings, then let $xZ = \{xy~:~y\in Z\}$.

For $x\in Z^n_2$, and for $1\leq i<j\leq n$, denote by $x[i,j]$ the binary string $x_ix_{i+1}\ldots x_j$.

We first present an integer function $\kappa(n)$.

\begin{definition}  Let $\kappa$ be the integer function defined as the following:
$$\kappa(n)=\left\{
\begin{array}{ll}
0,& \mbox{if $n=1$,}\\
\max\{1,\lceil \log_2n-2\log_2\log_2n\rceil\},  & \mbox{otherwise.}
\end{array}
\right.$$
\end{definition}
%$\kappa(n)=0$ if $n=1$, and $\kappa(n)=\max\{1,\lceil \log_2n-2\log_2\log_2n\rceil\}$ otherwise.
~~~~Next,  a permutation $\phi$ of binary strings is given as follows.

\begin{definition}
Assume $x\in Z^n_2$. Then $\phi(x)\in Z^n_2$ is the binary string such that

~~~~~~~~~~~~~~~~~~~~~~~~~~~~~~$\phi(x)[1,\kappa(n)] = x[1,\kappa(n)] \oplus x[n-\kappa(n) + 1, n]$,

~~~~~~~~~~~~~~~~~~~~~~~~~$\phi(x)[\kappa(n) + 1,n] = x[\kappa(n) + 1,n]$.
\end{definition}

Note that the restriction of $\phi$ to $Z^n_2$ is indeed a permutation of $Z^n_2$, with $\phi^2(x) = x$.

Now we give a recursive definition of the twisted hypercubes.

\begin{definition}
Set $H_1:=K_2$, with vertices 0 and 1. For $n\geq  2$, $H_n$ is obtained from two copies of $H_{n-1}$, $0H_{n-1}$ and $1H_{n-1}$, by adding edges connecting $0x$ and $1\phi(x)$ for all $x\in H_{n-1}$.
\end{definition}

The vertex set of $H_n$ is $Z^n_2$. It follows from the definition that $H_1=K_2$, $H_2=C_4$,  $H_3$ and $H_4$ are depicted in Figure 1.

\begin{figure}[!htb]
\centering
{\includegraphics[height=0.35\textwidth]{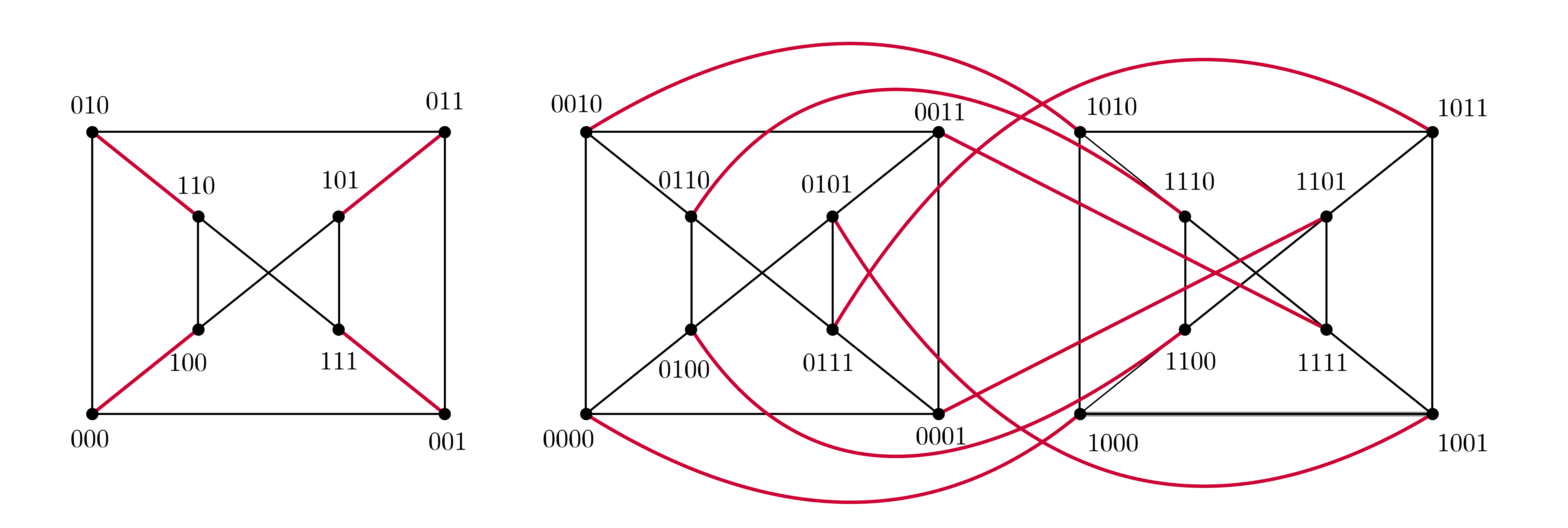}}

$H_3$~~~~~~~~~~~~~~~~~~~~~~~~~~~~~~~~~~~~~~~~~~~~~~~~~~~~$H_4$~~~~~~~~~~~~~~

Figure 1. $H_3$ and $H_4$
\end{figure}

By the definition of $H_n$, the following lemma is straightforward.

\begin{lemma}
\label{triangle-free}
$H_{n}$ is triangle-free.
\end{lemma}

It is seen that $H_n$ is a kind of $n$-dimensional hypercube-like
networks. Then $H_n$ is a $n$-regular graph with connectivity $n$. Furthermore, $H_n$ has the following properties.

\begin{lemma}  \cite{Liu}%Lemma 2.5
\label{common-neighbor}
For any $u,v\in V(H_n)$, $u$ and $v$ have at most two common neighbors.
\end{lemma}

\begin{lemma}  \cite{Xu10}%Lemma 2.6
\label{kappa1}
$\kappa_{1}(H_n)=2n-2$, where $n\geq3$.
\end{lemma}

\begin{lemma}  \cite{Chang13}%Lemma 2.7
\label{kappa2}
$\kappa_{2}(H_n)=3n-5$, where $n\geq5$.
\end{lemma}

Let $a\in \{0,1\}$ be an integer and $\overline{a}$ the complement of $a$, i.e., $\overline{a}=1-a$.

Let $u=u_1u_2\cdots u_n\in V(H_n)$. By the definition of the twisted hypercubes,  $u$ has only one neighbor $\overline{u_1}\phi(u_{2}u_{3} \cdots u_{n})$ in $\overline{u_1}H_{n-1}$,
and $n-1$ neighbors in $u_1H_{n-1}$. For every neighbor $w=w_1w_2\cdots w_n$ of $u$ in $u_1H_{n-1}$, there exists some $q~(1\leq q\leq n-1)$ such that $w_i=u_i$ for $1\leq i \leq q$. We use $u^i$ to denote the neighbor of $u$ with the same first $i-1$ bits as that of $u$, and $u^i$ is called the $(i)$-neighbor of $u$ , where $1\leq i\leq n$. That is, $u^i=u_1 \cdots u_{i-1} \overline{u_i}\phi(u_{i+1}\cdots u_n)=u_1 \cdots u_{i-1} \overline{u_i} (u_{i+1}\oplus u_{n-\kappa+1}) (u_{i+2}\oplus u_{n-\kappa+2}) \cdots (u_{i+\kappa}\oplus u_{n}) u_{i+\kappa+1} \cdots u_n$, where $\kappa=\kappa(n-i)$.

Note that there are $n-i$ neighbors of $u^{i}$ in $u_1 \cdots u_{i-1}\overline{u_i}H_{n-i}$, we use $u^{i,j}$ to denote the neighbor of $u^{i}$  with the same first $i+j-1$ bits as that of $u^i$, and $u^{i,j}$ is called the $(i,j)$-neighbor of $u$, where $1\leq j\leq n-i$. Denote $\kappa_i=\kappa(n-i)$. Then%u_1 \cdots u_{i-1} \overline{u_i}
$$u^{i,j}=\left\{
\begin{array}{ll}
u^i[1,i](u_{i+1}\oplus u_{n-\kappa_i+1})\cdots (u_{i+j-1}\oplus u_{n-\kappa_i+j-1})(\overline{u_{i+j}\oplus u_{n-\kappa_i+j}})\phi(u^i[{i+j+1},n]),\\~~~~~~~~~~~~~~~~~~~~~~~~~~~~~~~~~~~~~~~~~~~~~~
~~~~~~~~~~~~~~~~~~~~~~~~~~~~ \mbox{ \mbox{~if~} $1\leq j \leq \kappa_i$,}\\
u^i[1,i](u_{i+1}\oplus u_{n-\kappa_i+1})  \cdots (u_{i+\kappa_i}\oplus u_{n}) u_{i+\kappa_i+1} \cdots u_{i+j-1} \overline{u_{i+j}}\phi(u_{i+j+1} \cdots u_n),\\ ~~~~~~~~~~~~~~~~~~~~~~~~~~~~~~~~~~~~~~~~~~~~~~~
~~~~~~~~~~~~~~~~~~~~~~~~~~~~ \mbox{~if~$\kappa_i+1\leq j \leq n-i$.}\end{array}
\right.
$$
Specially,
$$u^{i,1}=\left\{
\begin{array}{ll}
u_1 \cdots u_{i-1}\overline{u}_{i} (\overline{u_{i+1}\oplus u_{n}})(u_{i+2}\oplus u_n)u_{i+3} \cdots u_{n},& \mbox{ \mbox{~if~} $\kappa_i=\kappa_{i+1}=1$,}\\
u_1 \cdots u_{i-1}\overline{u}_{i} (\overline{u_{i+1}\oplus u_{n-\kappa+1}}) (u_{i+2}\oplus u_{n-\kappa+2}\oplus u_{n-\kappa+1})\\
~~~~~\cdots (u_{i+\kappa}\oplus u_{n}\oplus u_{n-1})(u_{i+\kappa+1}\oplus u_{n}) u_{i+\kappa+2} \cdots u_{n},  & \mbox{~if~$\kappa_i=\kappa_{i+1}=\kappa\geq 2$,}\\
u_1 \cdots u_{i-1}\overline{u}_{i} (\overline{u_{i+1}\oplus u_{n-\kappa+1}}) u_{i+2} \cdots u_{n},  & \mbox{~if~$\kappa_i=\kappa_{i+1}+1=\kappa$.}\end{array}
\right.
$$

Similarly, there are $n-i-j$ neighbors of $u^{i,j}$ in $u_1 \cdots u_{i-1} \overline{u_i}(u_{i+1}\oplus u_{n-\kappa_i+1})\cdots (u_{i+j-1}\oplus u_{n-\kappa_i+j-1})(\overline{u_{i+j}\oplus u_{n-\kappa_i+j}})H_{n-i-j}$ if $1\leq j\leq \kappa_i$, in $u_1 \cdots u_{i-1} \overline{u_i} (u_{i+1}\oplus u_{n-\kappa_i+1})  \cdots (u_{i+\kappa_i}\oplus u_{n}) u_{i+\kappa_i+1} \cdots u_{i+j-1} \overline{u_{i+j}}H_{n-i-j}$ if  $\kappa_i+1\leq j \leq n-i$. We use $u^{i,j,k}$ to denote the neighbor of $u^{i,j}$ with the same first $i+j+k-1$ bits as that of $u^{i,j}$, and $u^{i,j,k}$ is called the $(i,j,k)$-neighbor of $u$, where $1\leq k\leq n-i-j$. Specially,
$$u^{i,1,1}=\left\{
\begin{array}{ll}
u_1 \cdots u_{i-1}\overline{u_{i}} (\overline{u_{i+1}\oplus u_{n}}) (\overline{u_{i+2}\oplus u_{n}})
 \phi(u_{i+3} \cdots u_{n}),
~~~~~~~ \mbox{ \mbox{~if~} $\kappa_i=\kappa_{i+1}=1$,}\\
u_1\cdots u_{i-1}\overline{u_{i}} (\overline{u_{i+1}\oplus u_{n-1}}) (\overline{u_{i+2}\oplus u_n\oplus u_{n-1}})
 \phi( (u_{i+3}\oplus u_n) u_{i+4}\cdots u_{n}),\\~~~~~~~~~~~~~~~~~~~~~~~~~~~~~~~~~~~~~
~~~~~~~~~~~~~~~~~~~~~~~~~~~~~~~~~\mbox{ \mbox{~if~} $\kappa_i=\kappa_{i+1}=2$,}\\
 u_1\cdots u_{i-1}\overline{u_{i}} (\overline{u_{i+1}\oplus u_{n-\kappa+1}}) (\overline{u_{i+2}\oplus u_{n-\kappa+2}\oplus u_{n-\kappa+1}})\phi(u^{i,1}[i+3,n]),\\ ~~~~~~~~~~~~~~~~~~~~~~~~~~~~~~~~~~~~~
 ~~~~~~~~~~~~~~~~~~~~~~~~~~~~~~~~~\mbox{ \mbox{~if~} $\kappa=\kappa_i=\kappa_{i+1}\ge 3$,}\\
u_1 \cdots u_{i-1}\overline{u_{i}} (\overline{u_{i+1}\oplus u_{n-\kappa+1}}) \overline{u_{i+2}} \phi(u_{i+3} \cdots u_{n}), ~~~~~~~~~~ \mbox{ \mbox{~if~} $\kappa_i=\kappa_{i+1}+1=\kappa. $}
\end{array}
\right.
$$
At the same time,
 $$u^{i,1,2}=\left\{
\begin{array}{ll}
u_1 \cdots u_{i-1}\overline{u_{i}} (\overline{u_{i+1}\oplus u_{n}}) (u_{i+2}\oplus u_{n})
\overline{ u_{i+3}}\phi(u_{i+4} \cdots u_{n}), ~~\mbox{ \mbox{~if~} $\kappa_i=\kappa_{i+1}=1$,}\\
u_1 \cdots u_{i-1}\overline{u_{i}} (\overline{u_{i+1}\oplus u_{n-1}}) (u_{i+2}\oplus u_n\oplus u_{n-1})(\overline{u_{i+3}\oplus u_n})\phi(u_{i+4}\cdots u_n),\\~~~~~~~~~~~~~~~~~~~~~~~~~~~~~~~~~~~~~
~~~~~~~~~~~~~~~~~~~~~~~~~~~~~~~~~\mbox{ \mbox{~if~} $\kappa_i=\kappa_{i+1}=2$,}\\
u_1\cdots u_{i-1}\overline{u_{i}} (\overline{u_{i+1}\oplus u_{n-2}}) (u_{i+2}\oplus u_{n-1}\oplus u_{n-2}) (\overline{u_{i+3}\oplus u_{n}\oplus u_{n-1}})\phi(u^{i,1}[i+4,n]) ,\\ ~~~~~~~~~~~~~~~~~~~~~~~~~~~~~~~~~~~~~
 ~~~~~~~~~~~~~~~~~~~~~~~~~~~~~~~~~\mbox{ \mbox{~if~} $\kappa_i=\kappa_{i+1}=3$,}\\
u_1\cdots u_{i-1}\overline{u_{i}} (\overline{u_{i+1}\oplus u_{n-\kappa+1}}) (u_{i+2}\oplus u_{n-\kappa+2}\oplus u_{n-\kappa+1}) (\overline{u_{i+3}\oplus u_{n-\kappa+3}\oplus u_{n-\kappa+2}})\\
~~\phi( (u_{i+4}\oplus u_{n-\kappa+4}\oplus u_{n-\kappa+3})\cdots (u_{i+\kappa}\oplus u_{n}\oplus u_{n-1})(u_{i+\kappa+1}\oplus u_{n})u_{i+\kappa+2} \cdots u_{n}),\\ ~~~~~~~~~~~~~~~~~~~~~~~~~~~~~~~~~~~~~
 ~~~~~~~~~~~~~~~~~~~~~~~~~~~~~~~~~\mbox{ \mbox{~if~} $\kappa=\kappa_i=\kappa_{i+1}\ge 4$,}\\
u_1 \cdots u_{i-1}\overline{u_{i}} (\overline{u_{i+1}\oplus u_{n-\kappa+1}}) u_{i+2} \overline{u_{i+3}}\phi(u_{i+4}\cdots u_{n}),~~~~~\mbox{ \mbox{~if~} $\kappa_i=\kappa_{i+1}+1=\kappa. $}
\end{array}
\right.
$$

Let $u^{h(\cdot)}$ be a vertex of $H_n$. If the first $i$ bits of $u^{h(\cdot)}$ is $u_1u_2\cdots u_i$, then $u^{h(\cdot)}$ has exactly one neighbor in  $u_1u_2 \cdots u_{i-1}\overline{u_i}H_{n-i}$, and we denote this neighbor $u^{h(\cdot),i^{*}}$. For example, let $u=0010\in H_4$, then $u^{1}=1010, u^{2}=0110, u^{3}=0000, u^{4}=0011, u^{1,1}=u^{2,1^{*}}=1110, u^{2,1}=0100, u^{3,1}=0001, u^{3,1^{*}}=1000, u^{4,1^{*}}=1111, u^{2,1,1^{*}}=1100, u^{3,1,1^{*}}=1101.$ Figure  2  illustrates the neighbors of $u$.

\begin{figure}[!htb]
\centering
{\includegraphics[height=0.6\textwidth]{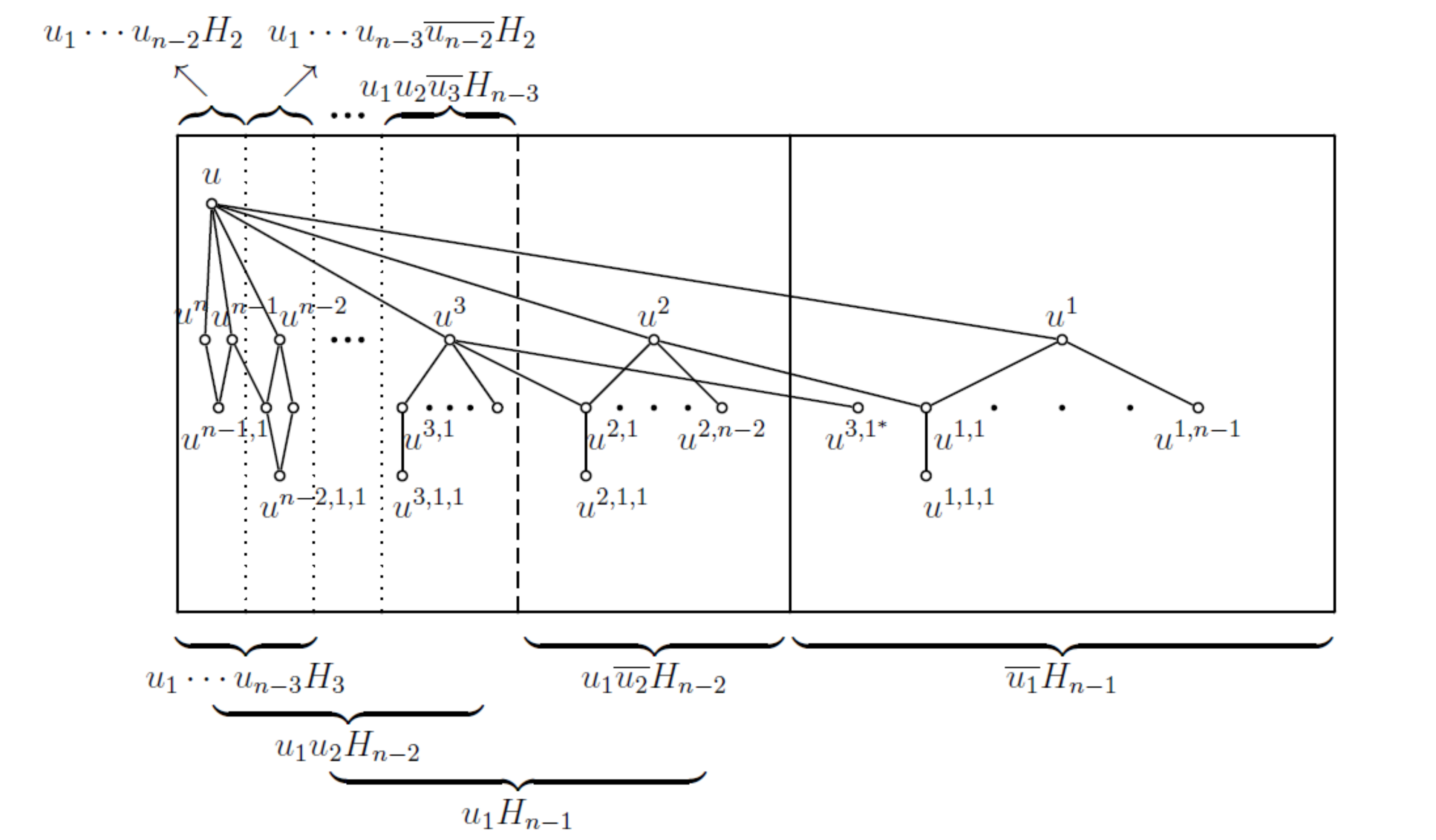}}

Figure 2. ~Some neighbors related to a vertex $u$ of $H_n$
\end{figure}

Note that $u^{i+1}=u_1u_2 \cdots u_{i} \overline{u_{i+1}} \phi(u_{i+2}u_{i+3} \cdots u_{n})$. Then
$$u^{i+1,i^*}=\left\{
\begin{array}{ll}
u_1 \cdots u_{i-1}\overline{u}_{i} (\overline{u_{i+1}\oplus u_{n}})(u_{i+2}\oplus u_n)u_{i+3} \cdots u_{n},& \mbox{ \mbox{~if~} $\kappa_i=\kappa_{i+1}=1$,}\\
u_1 \cdots u_{i-1}\overline{u}_{i} (\overline{u_{i+1}\oplus u_{n-\kappa+1}}) (u_{i+2}\oplus u_{n-\kappa+1}\oplus u_{n-\kappa+2})\\
~~~~~\cdots (u_{i+\kappa}\oplus u_{n-1}\oplus u_{n})(u_{i+\kappa+1}\oplus u_{n}) u_{i+\kappa+2} \cdots u_{n},  & \mbox{~if~$\kappa_i=\kappa_{i+1}=\kappa\geq 2$,}\\
u_1 \cdots u_{i-1}\overline{u}_{i} (\overline{u_{i+1}\oplus u_{n-\kappa+1}}) u_{i+2} \cdots u_{n},  & \mbox{~if~$\kappa_i=\kappa_{i+1}+1=\kappa$.}\end{array}
\right.
$$

It is seen that $u^{i+1,i^*}=u^{i,1}$. Then $u^{i,1}$ is a common neighbor of $u^i$ and $u^{i+1}$. Note that $u$ is also a common neighbor of $u^i$ and $u^{i+1}$. By Lemma~\ref{common-neighbor}, any two vertices have at most two common neighbors. So we have the following proposition.

\begin{proposition}
\label{neighbor}
Let $u=u_1u_2\cdots u_n$ be any vertex of $H_n$ and $N_{H_n}(u)=\{u^{1},u^{2},\ldots,u^{n}\}$. Then $N_{H_n}(u^{i})\cap N_{H_n}(u^{i+1})=\{u,u^{i,1} \}$, where $1\leq i \leq n-1$.
\end{proposition}

Recall that for $u=u_1\cdots u_n\in V(H_n)$, $u^{n}=u_1\cdots u_{n-1} \overline{u_n}$, $u^{n-1}=u_1\cdots u_{n-2}\overline {u_{n-1}} u_n$ and $u^{n-1,1}=u_1u_2 \cdots u_{n-2} \overline{u_{n-1}} ~\overline{u_n},$ we have
\begin{eqnarray*}u^{n,i^*}&=&u_1\cdots u_{i-1}\overline{u_i}\phi(u_{i+1}\cdots u_{n-1} \overline{u_n}),\\
u^{n-1,1,1^*}&=&\overline{u_1}\phi(u_2\cdots u_{n-2} \overline{u_{n-1}}~\overline{u_n}),\\
u^{n-1,1,2^*}&=&u_1\overline{u_2}\phi(u_3\cdots u_{n-2} \overline{u_{n-1}}~\overline{u_n}).
\end{eqnarray*}

The following two propositions are straightforward.

\begin{proposition}
\label{vertex}
For any $u=u_1u_2 \ldots u_n\in V(H_n)$, $N_{H_n}(u)=\{u^1,u^2,\ldots,u^n\}$, Then

(a) $u^{i,1}\neq u^{j,1}$ for $1\leq i \neq j \leq n-1$.

(b) $u^{n-1,1,1^{*}}\neq u^{1,1}$, $u^{n-1,1,1^{*}}\neq u^{1,1,1}$ and $u^{n-1,1,1^{*}}\neq u^{1,1,2}$.

(c) $u^{n-1,1,2^{*}}\neq u^{2,1}$, $u^{n-1,1,2^{*}}\neq u^{2,1,1}$ and $u^{n-1,1,2^{*}}\neq u^{2,1,2}$.

(d) $u^{n,i^{*}}\neq u^{i,1}$, $u^{n,i^{*}}\neq u^{i,1,1}$ and $u^{n,i^{*}}\neq u^{i,1,2}$, where $1\leq i \leq n-3$.
\end{proposition}

\begin{proposition}
\label{path}
For any $u=u_1u_2 \ldots u_n\in V(H_n)$, $N_{H_n}(u)=\{u^1,u^2,\ldots,u^n\}$, then
\begin{eqnarray}
u^{n,(n-2)^{*}}&\neq&u^{n-2,1},\nonumber
\\u^{n,(n-2)^{*},(n-3)^{*}}&\neq&u^{n-3,1},\nonumber
\\&\vdots&\nonumber
\\u^{n,(n-2)^{*},(n-3)^{*},\ldots,2^{*}}&\neq&u^{2,1},\nonumber
\\u^{n,(n-2)^{*},(n-3)^{*},\ldots,2^{*},1^{*}}&\neq&u^{1,1}.\nonumber
\end{eqnarray}
\end{proposition}

By Lemma~\ref{triangle-free} and Proposition~\ref{neighbor}, we have following proposition.

\begin{proposition}
\label{star}
For any $u=u_1u_2 \ldots u_n\in V(H_n)$, $N_{H_n}(u)=\{u^1,u^2,\ldots,u^n\}$,  we have

(a) $H_n[\{u^i,u^{i,1}, u^{i+1}, u^{i,1,1}\}]$ is a $K_{1,3}$ with center $u^{i,1}$, where $1\leq i \leq n-2$.

(b) $H_n[\{u^n,u^{n-1,1}, u^{n,1^*}, u^{n,2^*}\}]$ is a $K_{1,3}$ with center $u^{n}$.

(c) $H_n[\{u^{n-1},u^{n-1,1}, u^{n}, u^{n-1,1,1^*}\}]$ is a $K_{1,3}$ with center $u^{n-1,1}$.

(d) $H_n[\{u^i,u^{i,1}, u^{i+1}, u^{i,1,1}, u^{i,1,2}\}]$ is a $K_{1,4}$ with center $u^{i,1}$, where $1\leq i \leq n-3$.

(e) $H_n[\{u^n,u^{n-1,1}, u^{n,1^*}, u^{n,2^*}, u^{n,3^*}\}]$ is a $K_{1,4}$ with center $u^{n}$.

(f) $H_n[\{u^{n-1},u^{n-1,1}, u^{n}, u^{n-1,1,1^*}, u^{n-1,1,2^*}\}]$ is a $K_{1,4}$ with center $u^{n-1,1}$.

(g) $H_n[\{u^{n-2},u^{n-2,1}, u^{n-1}, u^{n-2,1,1}, u^{n-2,1,1^*}\}]$ is a $K_{1,4}$ with center $u^{n-2,1}$.
\end{proposition}

\section{$\kappa(H_{n}; K_{1,r})$ and $\kappa^{s}(H_{n}; K_{1,r})$ for $r\in\{3, 4\}$}

In this section, we first show some properties of $H_n$ related to stars, then explore the $T$-structure connectivity and $T$-substructure connectivity of $H_n$ for $T=K_{1,3}$ and $T=K_{1,4}$, respectively.

\begin{lemma}%Lemma 3.1
\label{star-u}
Let $K_{1,r}$ be a star in $H_{n}$. If $u$ is a vertex of $H_{n}-K_{1,r}$, then $|N_{H_{n}}(u) \cap V(K_{1,r})| \leq 2$.
\end{lemma}

\begin{proof}
Let $x$ be the center of the star $K_{1,r}$. If $(x,u)\in E(H_n)$, then $N_{H_{n}}(u) \cap V(K_{1,r})=\{x\}$ by Lemma~\ref{triangle-free}. That is, $|N_{H_{n}}(u) \cap V(K_{1,r})|=1$. If $(x,u)\notin E(H_n)$, then $|N_{H_{n}}(u) \cap V(K_{1,r})| \leq 2$ as $u$ and $x$ have at most two common neighbors by Lemma~\ref{common-neighbor}.
 \end{proof}

\begin{lemma}
\label{3-star-uv}
Let $K_{1,3}$ be a star in $H_{n}$. If $u$ and $v$ are two adjacent vertices of $H_{n}-K_{1,3}$, then $|N_{H_{n}}(\{u,v\}) \cap V(K_{1,3})| \leq 3$.
\end{lemma}

\begin{proof}
Let $V(K_{1,3})=\{x,x_1,x_2,x_3\}$ and $E(K_{1,3})=\{(x,x_1),(x,x_2),(x,x_3)\}$. If $(x,u)\in E(H_n)$, then $N_{H_{n}}(u) \cap V(K_{1,3})=\{x\}$ by Lemma~\ref{triangle-free}. That is, $|N_{H_{n}}(u) \cap V(K_{1,3})|=1$. By Lemma~\ref{star-u}, $|N_{H_{n}}(v) \cap V(K_{1,3})|\leq 2$.
Then $|N_{H_{n}}(\{u,v\}) \cap V(K_{1,3})| \leq 3$. Therefore $(x,u)\notin E(H_n)$. By the symmetry of $u$ and $v$, $(x,v)\notin E(H_n)$.
Then $N_{H_{n}}(\{u,v\}) \cap V(K_{1,3}) \subseteq\{x_1,x_2,x_3\}$ and thus $|N_{H_{n}}(\{u,v\}) \cap V(K_{1,3})| \le 3$.
 \end{proof}

\begin{lemma}%Lemma 3.3
\label{4-star-uv}
Let $K_{1,r}$ be a star in $H_{n}$ with center $x$. If $u$ and $v$ are two adjacent vertices of $H_{n}-K_{1,r}$, then $|N_{H_{n}}(\{u,v\}) \cap V(K_{1,r})| \leq 4$. Moreover, if $|N_{H_{n}}(\{u,v\}) \cap V(K_{1,r})|=4$, then $(u,x),(v,x)\notin E(H_n)$ and $u,v,x\in V(aH_{n-1})$ for some $a\in \{0,1\}$.
\end{lemma}

\begin{proof}
By Lemma~\ref{star-u}, $|N_{H_{n}}(u) \cap V(K_{1,r})| \leq 2$ and $|N_{H_{n}}(v) \cap V(K_{1,r})| \leq 2$. Then $|N_{H_{n}}(\{u,v\}) \cap V(K_{1,r})| \leq 4$.

Suppose that $|N_{H_{n}}(\{u,v\}) \cap V(K_{1,r})| = 4$, in the following, we show that $u,v,x\in V(aH_{n-1})$ for some $a\in \{0,1\}$.
Note that $|N_{H_{n}}(u) \cap V(K_{1,r})|=2$ and $|N_{H_{n}}(v) \cap V(K_{1,r})|=2$, then $(u,x),(v,x)\notin E(H_n)$. If $u\in V(0H_{n-1})$ and $v\in V(1H_{n-1})$, then  $x$ and $v$ have at most one common neighbor if $x\in V(0H_{n-1})$, and $x$ and $u$ have at most one common neighbor if $x\in V(1H_{n-1})$, a contradiction. Therefore  $u,v\in V(aH_{n-1})$ for some $a\in \{0,1\}$. If $x\in V(\overline{a}H_{n-1})$, then $N_{H_{n}}(\{u,v\})\cap N_{H_{n}}(x)\subseteq \{u^1,v^1,x^1\}$, i.e, $|N_{H_{n}}(\{u,v\}) \cap V(K_{1,r})|\leq 3$, a contradiction. Therefore $x\in V(aH_{n-1})$.
 \end{proof}

\begin{lemma}%Lemma 3.4
\label{4'-star-uv}
Let $F_1,F_2,\ldots,F_{\lceil n/2\rceil-1}$ be $\lceil\frac{n}{2}\rceil-1$ stars in $H_{n}$.
If $u$ and $v$ are two adjacent vertices of $H_{n}-\cup_{i=1}^{\lceil n/2\rceil-1}V(F_i)$, then $| N_{H_n}(\{u,v\}) \cap (\cup_{i=1}^{\lceil n/2\rceil-1}V(F_i))| \leq 2n-3$.
\end{lemma}

\begin{proof}
Note that $|N_{H_n}(\{u,v\})|=2n-2$, suppose to the contrary that $| N_{H_n}(\{u,v\}) \cap (\cup_{i=1}^{\lceil n/2\rceil-1} V(F_i))|= 2n-2$.
By Lemma~\ref{4-star-uv}, $|N_{H_{n}}(\{u,v\}) \cap V(K_{1,r})| \leq 4$, then $2n-2=| N_{H_n}(\{u,v\}) \cap (\cup_{i=1}^{\lceil n/2\rceil-1}V(F_i))| \leq 4(\lceil\frac{n}{2}\rceil-1)$. Thus $n$ is odd, $|N_{H_{n}}(\{u,v\}) \cap V(F_i)| =4$ for $1\leq i\leq \lceil n/2\rceil-1$ and $ N_{H_n}(\{u,v\}) \cap (\cup_{i=1}^{\lceil n/2\rceil}V(F_i))=N_{H_n}(\{u,v\})$.
Let $x_i$ be the center of $F_i$, then $(u,x_i),(v,x_i)\notin E(H_n)$ and $u,v,x_i\in V(aH_{n-1})$ for some $a\in \{0,1\}$ by Lemma~\ref{4-star-uv}.  Assume, without loss of generality, that $u^1\in F_1$. Then $u^1\neq x_1$ as $(u,x_1)\notin E(H_n)$. On the other hand, $(u^1,x_1)\in E(H_n)$, which implies $x_1\in V(\overline{a}H_{n-1})$, a contradiction.
\end{proof}

\begin{lemma}%Lemma 3.5$\kappa(H_{n}; K_{1,r})\geq \lceil\frac{n}{2}\rceil$ and
\label{lower}
$\kappa^{s}(H_{n}; K_{1,r})\geq \lceil\frac{n}{2}\rceil$ for $3\leq r\leq 4$ and $n\geq 4$.
\end{lemma}

\begin{proof}
Let $\mathcal{F}=\{\underbrace{P_1,\ldots,P_1}_{x_1}, \underbrace{P_2,\ldots,P_2}_{x_2},\underbrace{P_3,\ldots,P_3}_{x_3},\underbrace{K_{1,3},\ldots,K_{1,3}}_{x_4}\}$ if $r=3$, and $\mathcal{F}=\{\underbrace{P_1,\ldots,P_1}_{x_1}, \underbrace{P_2,\ldots,P_2}_{x_2},\underbrace{P_3,\ldots,P_3}_{x_3},\underbrace{K_{1,3},\ldots,K_{1,3}}_{x_4},\underbrace{K_{1,4},\ldots,K_{1,4}}_{x_5}\}$ if $r=4$, where $x_i\geq 0$ and $1\leq i\leq r+1$.  Then $|\mathcal{F}| =\sum^{r+1}_{i=1}x_i $.
Suppose to the contrary that $|\mathcal{F}| \leq \lceil\frac{n}{2}\rceil-1$ and $H_n-\mathcal{F}$ is disconnected, then $H_n-\mathcal{F}$ has at least two components. Let $C$ be the smallest component of $H_n-\mathcal{F}$. We consider the following three cases.

\vskip .2cm
{\bf Case 1.} $|V(C)| =1$.

In this case, $C$ is an isolated vertex $w$. Note that $|N_{H_n}(w)|=n$. By Lemmas~\ref{triangle-free} and \ref{star-u}, every element in $\mathcal{F}$ contains at most two neighbors of $w$, then $2|\mathcal{F}|\geq|N_{H_n}(w)|$. Thus, $2(\lceil\frac{n}{2}\rceil-1)\geq n$, a contradiction.

\vskip .2cm
{\bf Case 2.} $|V(C)| =2$.

In this case, $C$ is an edge $(u,v)$. Note that $N_{H_n}(\{u, v\}) =2n-2$. If $r=3$, every element in $\mathcal{F}$ contains at most three neighbors of $\{u, v\}$ by Lemmas~\ref{triangle-free} and~\ref{3-star-uv}, then $3|\mathcal{F}|\geq|N_{H_n}(\{u, v\})|$. Thus $3(\lceil\frac{n}{2}\rceil-1)\geq 2n-2$, a contradiction. Therefore $r=4$. By Lemma~\ref{4'-star-uv}, all elements in $\mathcal{F}$ contains at most $2n-3$ neighbors of $\{u, v\}$, a contradiction.

\vskip .2cm
{\bf Case 3.} $|V(C)| \geq3$.

Note that $|\mathcal{F}| \leq \lceil\frac{n}{2}\rceil-1$ and every element in $\mathcal{F}$ contains at most five vertices, then $|V(\mathcal{F})| \leq 5(\lceil\frac{n}{2}\rceil-1)$. If $n=4$, then $|V(\mathcal{F})|\leq 5$. By Lemma~\ref{kappa1}, $\kappa_1(H_n) =2n-2=6$, which implies that we have to delete at least $6$ vertices to separate $C$ from $H_4$, a contradiction. Therefore $n\geq 5$. By Lemma~\ref{kappa2}, $\kappa_2(H_n) =3n-5$, which implies that we have to delete at least $3n-5>5(\lceil\frac{n}{2}\rceil-1)$ vertices to separate $C$ from $H_n$, a contradiction.

Thus, $\kappa^{s}(H_{n}; K_{1,3})\geq \lceil\frac{n}{2}\rceil$.
\end{proof}%$\kappa(H_{n}; K_{1,3})\geq \lceil\frac{n}{2}\rceil$ and

\subsection{$\kappa(H_{n}; K_{1,3})$ and $\kappa^{s}(H_{n}; K_{1,3})$}

\vskip.2cm

\begin{lemma}%Lemma 3.6 and $\kappa^{s}(H_{n}; K_{1,3})\leq \lceil\frac{n}{2}\rceil$
\label{3-upper}
$\kappa(H_{n}; K_{1,3})\leq \lceil\frac{n}{2}\rceil$ for $n\geq 4$.
\end{lemma}

\begin{proof}
For any $u=u_1\cdots u_n\in V(H_n)$, $N_{H_n}(u)=\{u^1,\ldots,u^n\}$. Let $T_i$ be the subgraph induced  by $\{u^i,u^{i,1}, u^{i+1}, u^{i,1,1}\}$ for $1\leq i \leq n-2$. If $n$ is odd, we let $T_n$ be the subgraph induced  by $\{u^n, u^{n-1,1},u^{n,1^{*}}, u^{n,2^{*}}\}$; and if $n$ is even, let $T_{n-1}$ the subgraph induced  by $\{u^{n-1}, u^{n-1,1}, u^{n}, u^{n-1,1,1^{*}}\}$. Then $T_i$ is a $K_{1,3}$ by Proposition~\ref{star}(a-c) for $1\leq i\leq n$, and $V(T_i)\cap V(T_j)=\emptyset$ by Proposition~\ref{vertex} for $1\leq i\neq j\leq n$.

Set $\mathcal{S} = \{T_1, T_3,\ldots, T_{n-2}, T_n\}$ if $n$ is odd; and $\mathcal{S} = \{T_1, T_3,\ldots, T_{n-3}, T_{n-1}\}$ if $n$ is even.
Then, in either case, $H_n-\mathcal{S}$ is disconnected, one component is $\{u\}$ and $|\mathcal{S}| =\lceil\frac{n}{2} \rceil$. Thus, $\kappa(H_{n}; K_{1,3})\leq \lceil\frac{n}{2}\rceil$.% and $\kappa^{s}(H_{n}; K_{1,3})\leq \lceil\frac{n}{2}\rceil$.
\end{proof}

Note that $\kappa^{s}(H_{n}; K_{1,3})\leq \kappa(H_{n}; K_{1,3})$, and thus, by Lemmas \ref{lower} and \ref{3-upper}, we have the following result.

\begin{theorem}
$\kappa(H_{n}; K_{1,3})=\kappa^{s}(H_{n}; K_{1,3})=\lceil\frac{n}{2}\rceil$ for $n\geq 4$.
\end{theorem}

\subsection{$\kappa(H_{n}; K_{1,4})$ and $\kappa^{s}(H_{n}; K_{1,4})$}

\begin{lemma}%Lemma 3.8
\label{4-upper}
$\kappa(H_{n}; K_{1,4})\leq \lceil\frac{n}{2}\rceil$% and $\kappa^{s}(H_{n}; K_{1,4})\leq \lceil\frac{n}{2}\rceil$
 for $n\geq4$.
\end{lemma}

\begin{proof}
For any $u\in V(H_n)$, $N_{H_n}(u)=\{u^1,\ldots,u^n\}$. Let $T_i$ be the subgraph induced  by $\{u^i,u^{i,1}, u^{i+1}, u^{i,1,1},  u^{i,1,2}\}$ for $1\leq i \leq n-3$. If $n$ is odd, we let $T_{n-2}$ be the subgraph induced by $\{u^{n-2},u^{n-2,1}, u^{n-1}, u^{n-2,1,1}, u^{n-2,1,1^*}\}$ and $T_n$ the subgraph induced by $\{u^n, u^{n-1,1},u^{n,1^{*}}, u^{n,2^{*}}, u^{n,3^{*}}\}$; and if $n$ is even, let $T_{n-1}$ be the subgraph induced by $\{u^{n-1},$ $u^{n-1,1}, u^{n}, u^{n-1,1,1^{*}}, u^{n-1,1,2^{*}}\}$. Then $T_i$ is a $K_{1,4}$ by Proposition~\ref{star}(d-g) for $1\leq i\leq n$, and $V(T_i)\cap V(T_j)=\emptyset$ by Proposition~\ref{vertex} for $1\leq i\neq j\leq n$.

Set $\mathcal{S} = \{T_1, T_3,\ldots,T_{n-4}, T_{n-2}, T_n\}$ if $n$ is odd; and $\mathcal{S} = \{T_1, T_3,\ldots, T_{n-3}, T_{n-1}\}$ if $n$ is even.
Then, in either case, $H_n-\mathcal{S}$ is disconnected, one component is $\{u\}$ and $|\mathcal{S}| =\lceil\frac{n}{2} \rceil$. Thus, $\kappa(H_{n}; K_{1,4})\leq \lceil\frac{n}{2}\rceil$ and $\kappa^{s}(H_{n}; K_{1,4})\leq \lceil\frac{n}{2}\rceil$.
\end{proof}

Note that $\kappa^{s}(H_{n}; K_{1,4})\leq \kappa(H_{n}; K_{1,4})$, and hence, by Lemmas \ref{lower} and \ref{4-upper}, we have the following result.

\begin{theorem}
$\kappa(H_{n}; K_{1,4})=\kappa^{s}(H_{n}; K_{1,4})=\lceil\frac{n}{2}\rceil$ for $n\geq 4$.
\end{theorem}

%%%%%%%%%%%%4
\section{$\kappa(H_{n}; P_{k})$ and $\kappa^{s}(H_{n}; P_{k})$}

Recall that $\kappa(H_n; P_1)=\kappa^{s}(H_{n}; P_{1})=\kappa(H_n)=n$ for $n\geq3$. So we assume $k\geq 2$.

\subsection{$\kappa(H_{n}; P_{2})$ and $\kappa^{s}(H_{n}; P_{2})$}

\begin{lemma}%Lemma 4.1
\label{book}
If $\kappa(n-1)\geq 2$, then $H_n$ does not consists of $n-1$ quadrilaterals sharing a common edge.
\end{lemma}

\begin{proof}
For any $u=u_1u_2 \ldots u_n \in V(H_n)$, we show that for any $1\leq i\leq n-2$, $u^i$ and $u^n$ have no other common neighbors other than $u$,
and $u^{n-1}$ and $u^1$ have no other common neighbors other than $u$.

Recall that $u^{i}=u_1 \cdots u_{i-1} \overline{u_i} (u_{i+1}\oplus u_{n-\kappa_i+1}) \cdots (u_{i+\kappa_i}\oplus u_{n})u_{i+\kappa_i+1} \cdots u_{n-1}u_n$, where $\kappa_i=\kappa(n-i)$.

First we assume $i\leq n-2$. We only need to show that $(u^i)^j\neq (u^n)^k$ for $1\leq j\neq i\leq n$ and $1\leq k\leq n-1$. If  $1\leq j\neq i\leq n-1$, then $(u^i)^j[n]=u_n\neq \overline{u_n}=(u^n)^k[n]$, thus $(u^i)^j\neq (u^n)^k$. Therefore $j=n$. If $i\neq k$, then $(u^i)^n[i]\neq (u^n)^k[i]$; and if $i=k$, then $(u^i)^n[i+\kappa_i]=u_{i+\kappa_i}\oplus u_n\neq u_{i+\kappa_i}\oplus \overline{u_n}=(u^n)^i[i+\kappa_i]$. Thus $(u^i)^n\neq (u^n)^k$.

Now we assume $i=n-1$. We only need to show that $(u^1)^{p}\neq (u^{n-1})^q$ for $2\leq p\leq n$ and $1\leq q\neq n-1\leq n$. If $2\leq p\leq n-2$, then
$(u^1)^p[n-1,n]=u_{n-1}u_n\neq \overline{u_{n-1}}u_n=(u^{n-1})^q[n-1,n]$ for $1\leq q\leq n-2$ and $(u^1)^p[n-1,n]=u_{n-1}u_n\neq\overline{u_{n-1}} ~\overline{u_n}=(u^{n-1})^n[n-1,n]$, and thus $(u^1)^{p}\neq (u^{n-1})^q$. If $p=n$, then $(u^1)^n[n-1,n]=u_{n-1}\overline{u_n}\neq \overline{u_{n-1}}u_n=(u^{n-1})^q[n-1,n]$ for $1\leq q\leq n-2$ and $(u^1)^n[n-1,n]=u_{n-1}\overline{u_n}\neq\overline{u_{n-1}} ~\overline{u_n}=(u^{n-1})^n[n-1,n]$, and thus $(u^1)^{n}\neq (u^{n-1})^q$. Therefore $p=n-1$. If $q\neq 1$, then $(u^1)^{n-1}[1]\neq (u^{n-1})^q[1]$; and if $q=1$, then
$(u^{n-1})^1=\overline{u_1}(u_2\oplus u_{n-\kappa_1+1})\cdots (u_{\kappa_1-1}\oplus u_{n-2})(u_{\kappa_1}\oplus \overline{u_{n-1}})(u_{\kappa_1+1}\oplus u_{n})u_{\kappa_1+2}\cdots u_{n-2}\overline{u_{n-1}}u_n$ as $\kappa_1=\kappa(n-1)\ge 2$ and thus $(u^1)^{n-1}[\kappa_1]=(u_{\kappa_1}\oplus u_{n-1})\neq (u_{\kappa_1}\oplus \overline{u_{n-1}})
=(u^{n-1})^{1}[\kappa_1]$. Hence $(u^1)^{n-1}\neq (u^{n-1})^q$.
\end{proof}

\begin{lemma}%Lemma 4.2 and $\kappa^{s}(H_{n}; P_2)\leq n-1$ and $\kappa^{s}(H_{n}; P_2)\leq n$
\label{P2-upper}
$\kappa(H_{n}; P_2)\leq n-1$ if $\kappa(n-1)=1$ and $\kappa(H_{n}; P_2)\leq n$ if $\kappa(n-1)\geq 2$.
\end{lemma}

\begin{proof}
we distinguish cases pertaining to the value of $\kappa(n)$ in the following.

\vskip .2cm
{\bf Case 1.} $\kappa(n-1)=1$.

In this case, we set $u=00 \cdots 0$. Then $u^{i}=0 \cdots 0 u_i 0 \cdots 0$, where $u_i=1, 1\leq i \leq n$. Let $v=u^{n-1}=0 \cdots 010$, then $v^{i}=0 \cdots 0 v_i 0 \cdots 010$, where $v_i=1, 1\leq i \leq n-2$; $v^{n}=0 \cdots 011$. Note that if $\kappa(n-1)=1$, then $(u^{i},v^{i})\in E(H_n)$; and hence $C_4=\langle u,u^{i},v^{i},v,u\rangle$ is a cycle of length 4.

Set $\mathcal{S} = \{u^iv^{i}: 1\leq i \leq n$ and $i\neq n-1\}$.
Then, $H_n-\mathcal{S}$ is disconnected, one component is $\{u,v\}$ and $|\mathcal{S}| =n-1$. Thus, $\kappa(H_{n}; P_2)\leq n-1$.% and consequently $\kappa^{s}(H_{n}; P_2)\leq n-1$.

\vskip .2cm
{\bf Case 2.} $\kappa(n-1)\geq 2$.

\vskip .2cm

In this case, for any $u=u_1u_2 \ldots u_n\in V(H_n)$, $N_{H_n}(u)=\{u^1,u^2,\ldots,u^n\}$. By Lemma~\ref{book}, $H_n$ does not consists of $n-1$ quadrilaterals sharing a common edge. So we can set that $T_i$ be the subgraph induced by $\{u^i,u^{i,1}\}$ for $1\leq i \leq n-1$, then $T_i$ is isomorphic to $P_2$.

Set $\mathcal{S} = \{T_1, T_2,\ldots, T_{n-1}, T_n\}$, where $T_n$ is the subgraph induced by $\{u^n, u^{n,1^{*}}\}$.
Then, $H_n-\mathcal{S}$ is disconnected, one component is $\{u\}$ and $|\mathcal{S}| =n$. Thus, $\kappa(H_{n}; P_2)\leq n$.% and consequently $\kappa^{s}(H_{n}; P_2)\leq n$.
\end{proof}

\begin{lemma}%Lemma 4.3$\kappa(H_{n}; P_2)\geq n-1$ and $\kappa(H_{n}; P_2)\geq n$ and
\label{P2-lower}
$\kappa^{s}(H_{n}; P_2)\geq n-1$ if $\kappa(n-1)=1$ and $\kappa^{s}(H_{n}; P_2)\geq n$ if $\kappa(n-1)\geq 2$.
\end{lemma}

\begin{proof}
we distinguish cases pertaining to the value of $\kappa(n)$ in the following.

\vskip .2cm
{\bf Case 1.} $\kappa(n-1)=1$.

Let $\mathcal{F}=\{\underbrace{P_1,\ldots,P_1}_{x_1}, \underbrace{P_2,\ldots,P_2}_{x_2}\}$ and $|\mathcal{F}| =x_1+x_2\leq n-2$ for $x_1, x_2\geq 0$. Suppose to the contrary that $H_n-\mathcal{F}$ is disconnected, then $H_n-\mathcal{F}$ has at least two components. Let $C$ be the smallest component of $H_n-\mathcal{F}$. We consider following two cases.

\vskip .2cm
{\bf Subcase 1.1.} $|V(C)| =1$.

In this subcase, we set $V(C) =\{w\}$. Note that $|N_{H_n}(w)|=n$. By Lemma~\ref{triangle-free}, every element in $\mathcal{F}$ contains at most one neighbor of $w$. Thus, we have to delete at least $n$ elements of $\mathcal{F}$ to isolate $C$. But it is impossible since $|\mathcal{F}|\leq n-2<n$.

\vskip .2cm
{\bf Subcase 1.2.} $|V(C)| \geq2$.

In this subcase, by Lemma~\ref{kappa1}, $\kappa_{1}(H_n)=2n-2$. This implies that we have to delete at least $2n-2$ vertices to isolate $C$. Since $|\mathcal{F}|\leq n-2$, we have $|V(\mathcal{F})|\leq 2(n-2) =2n-4 <2n-2$, a contradiction. Thus, $\kappa^{s}(H_{n}; P_2)\geq n-1$ and $\kappa(H_{n}; P_2)\geq n-1$.

\vskip .2cm
{\bf Case 2.} $\kappa(n-1)\geq 2$.

In this case, $n>5$. Let $\mathcal{F}=\{\underbrace{P_1,\ldots,P_1}_{x_1}, \underbrace{P_2,\ldots,P_2}_{x_2}\}$ and $|\mathcal{F}| =x_1+x_2\leq n-1$ for $x_1, x_2\geq0$. Suppose to the contrary that $H_n-\mathcal{F}$ is disconnected, then $H_n-\mathcal{F}$ has at least two components. Let $C$ be the smallest component of $H_n-\mathcal{F}$. We consider following three cases.

\vskip .2cm
{\bf Subcase 2.1.} $|V(C)| =1$.

In this subcase, we set $V(C) =\{w\}$. Note that $|N_{H_n}(w)|=n$. By Lemma~\ref{triangle-free}, every element in $\mathcal{F}$ contains at most one neighbor of $w$. Thus, we have to delete at least $n$ elements of $\mathcal{F}$ to isolate $C$. But it is impossible since $|\mathcal{F}|\leq n-1<n$.

\vskip .2cm
{\bf Subcase 2.2.} $|V(C)| =2$.

In this subcase, $C$ is an edge of $H_n$ and $|N_{H_n}(C)|=2n-2$.
By Lemma~\ref{kappa1} and  Lemma~\ref{book}, we have to delete at least $\frac{2n-2}{2}+1=n$ elements of $\mathcal{F}$ to separate $C$. However, $|\mathcal{F}| \leq n-1< n$, a contradiction.

\vskip .2cm
{\bf Subcase 2.3.} $|V(C)| \geq 3$.

In this subcase, by Lemma~\ref{kappa2}, $\kappa_{2}(H_n)=3n-5$. This implies that we have to delete at least $3n-5$ vertices to isolate $C$. Since $|\mathcal{F}|\leq n-1$, we have $|V(\mathcal{F})|\leq 2(n-1) =2n-2 <3n-5$, a contradiction with $n> 5$. Thus, $\kappa^{s}(H_{n}; P_2)\geq n$ and $\kappa(H_{n}; P_2)\geq n$.
\end{proof}

Note that $\kappa^{s}(H_{n};P_2)\leq \kappa(H_{n}; P_2)$, and hence, by Lemmas~\ref{P2-upper} and~\ref{P2-lower}, we have the following result.

\begin{theorem}
$\kappa(H_{n}; P_2) =\kappa^{s}(H_{n}; P_2)= n-1$ if $\kappa(n-1)=1$; $\kappa(H_{n}; P_2) =\kappa^{s}(H_{n}; P_2)= n$ if $\kappa(n-1)\geq 2$.
\end{theorem}

\subsection{$\kappa(H_{n}; P_{k})$ and $\kappa^{s}(H_{n}; P_{k})$ with $k\geq 3$}

\begin{lemma}
\label{v-p}
Let $P_{k}$ be a path of order $k$ in $H_{n}$ with $1 \leq k \leq n$. If $v$ is a vertex of $H_{n}-P_{k}$, then $|N_{H_{n}}(v) \cap V(P_{k})| \leq \lceil \frac{k}{2} \rceil$.
\end{lemma}

\begin{proof}
By Lemma~\ref{triangle-free}, $v$ can be adjacent to at most one vertex of any two consecutive vertices on $P_k$. Thus, the lemma follows.
\end{proof}

\begin{lemma}
\label{uv-P}
Let $P_{k}$ be a path of order $k$ in $H_{n}$ with $3 \leq k \leq n$. If $u$ and $v$ are two adjacent vertices of $H_{n}-P_{k}$, then $|N_{H_{n}}(\{u,v\}) \cap V(P_{k})| \leq 2\lfloor\frac{k}{3}\rfloor + (k(mod ~3)) $. Moreover, $|N_{H_{n}}(\{u,v\}) \cap V(P_{k})|\leq k-1$.
\end{lemma}

\begin{proof}
Let $P_k=\langle v_1, v_2, \ldots, v_k\rangle$. We first show that

~~~~~~~~~~~~~~~~~~~~$|N_{H_{n}}(\{u,v\}) \cap \{v_{i-1}, v_i, v_{i+1} \}| \leq 2$ for $2\leq i \leq k-1$.~~~~~~~~~~~~~~~~~~~($*$)\\
 Otherwise, we suppose that $|N_{H_{n}}(\{u,v\}) \cap \{v_{i-1}, v_i, v_{i+1} \}| \geq 3 $ for some $i ~(2\leq i \leq k-1)$. Without loss of generality, we assume that $|N_{H_{n}}(u) \cap \{v_{i-1}, v_i, v_{i+1} \}| \geq 2$. By Lemma~\ref{triangle-free}, we have $uv_{i-1}, uv_{i+1}\in E(H_n)$, and $uv_i\notin E(H_n) $. Moreover, $|N_{H_{n}}(u) \cap \{v_{i-1}, v_i, v_{i+1} \}| = 2$. Then $|N_{H_{n}}(v) \cap \{v_{i-1}, v_i, v_{i+1} \}| \geq 1$ and $vv_{i-1}, vv_{i+1}\notin E(H_n)$ by Lemma~\ref{triangle-free}. So $vv_i\in E(H_n)$, and thus $\{v, v_{i-1}, v_{i+1} \}\subseteq N(u)\cap N(v_i)$, a contradiction with Lemma~\ref{common-neighbor}.

Now we prove this lemma by induction on $k$. For $k=3$, by the above proof, $|N_{H_{n}}(\{u,v\}) \cap V(P_{3})|\leq 2$. In the induction step, assume that the lemma is true for $3\leq k\le  l$.  When $k=l+1$, $|N_{H_n}(v_{l+1} )\cap \{u,v\}|\le 1$ by Lemma~\ref{triangle-free}, and then $|N_{H_{n}}(\{u,v\}) \cap V(P_{l+1})|\le |N_{H_{n}}(\{u,v\}) \cap V(P_{l})|+1 \leq 2\lfloor\frac{l}{3}\rfloor + (l(\mbox{mod} ~3))+1$ by the
inductive hypothesis.

Note that if $(l+1)(\mbox{mod} ~3)=1$ or 2, then $(l+1)(\mbox{mod} ~3)=l(\mbox{mod} ~3)+1$ and $\lfloor\frac{l}{3}\rfloor=\lfloor\frac{l+1}{3}\rfloor$, and hence $2\lfloor\frac{l}{3}\rfloor + (l(\mbox{mod} ~3))+ 1 = 2\lfloor\frac{l+1}{3}\rfloor + ((l+1)(\mbox{mod} ~3))$. So, in the following, we assume that $l\ge 5$ and $(l+1)(\mbox{mod} ~3)=0$. Then $(l-2)(\mbox{mod} ~3)=0$. By ($*$),
$|N_{H_{n}}(\{u,v\}) \cap V(P_{l+1})|\le |N_{H_{n}}(\{u,v\}) \cap V(P_{l-2})|+2 \leq 2\lfloor\frac{l-2}{3}\rfloor +2= 2\lfloor\frac{l+1}{3}\rfloor$.

Therefore the proof of the lemma is complete.
\end{proof}

\begin{lemma}%Lemma4.7and $\kappa^{s}(H_{n}; P_{k})\leq \lceil\frac{2n}{k+1}\rceil$ and $\kappa^{s}(H_{n}; P_{k})\leq \lceil\frac{2n}{k}\rceil$
\label{Pk-upper}
Let $3 \leq k \leq n$. Then
$\kappa(H_{n}; P_{k})\leq \lceil\frac{2n}{k+1}\rceil$  if $k$ is odd and $\kappa(H_{n}; P_{k})\leq \lceil\frac{2n}{k}\rceil$ if $k$ is even.
\end{lemma}

\begin{proof}
For any $u=u_1u_2 \ldots u_n\in V(H_n)$, $N_{H_n}(u)=\{u^1,u^2,\ldots,u^n\}$, we distinguish cases pertaining to the parity of $k$ in the following.

\vskip .2cm
{\bf Case 1.} $k$ is odd.

In this case, we denote $t:=\frac{k+1}{2}$. Let $n=qt+r$ for some nonnegative integers $q$ and $r$ with $0 \leq r \leq \frac{k-1}{2}$. Since $n \geq k$, we have $q\geq1$.

\vskip .2cm
{\bf Subcase 1.1.} $r=0$.

In this subcase, we set

~~~~~~~~$P^{1}=\langle u^{1},u^{1,1},u^{2},\ldots,u^{t-1},u^{t-1,1},u^{t} \rangle$,

~~~~~~~~$P^{2}=\langle u^{t+1},u^{t+1,1},u^{t+2},\ldots,u^{k},u^{k,1},u^{k+1} \rangle$,

~~~~~~~~~~ $\vdots$

~~~~~~~~$P^{q}=\langle u^{n-t+1},u^{n-t+1,1},u^{n-t+2},\ldots,u^{n-1},u^{n-1,1},u^{n} \rangle$.\\
Then $V(P^i)\cap V(P^j)=\emptyset$ by Proposition~\ref{path} for $1\le i\neq j\le q$, and then $\mathcal{F}=\{P^{1},P^{2},\ldots,P^{q}\}$ forms a $P_k$-structure-cut of $H_n$, one component of $H_n-\mathcal{F}$ is $\{u\}$ and $|\mathcal{F}|=\frac{2n}{k+1}$.

\vskip .2cm
{\bf Subcase 1.2.} $1\leq r \leq \frac{k-1}{2}$.

In this subcase, we set

~~~~~~~~$P^{1}=\langle u^{1},u^{1,1},u^{2},\ldots,u^{t-1},u^{t-1,1},u^{t} \rangle$,

~~~~~~~~$P^{2}=\langle u^{t+1},u^{t+1,1},u^{t+2},\ldots,u^{k},u^{k,1},u^{k+1} \rangle$,

~~~~~~~~~~ $\vdots$

~~~~~~~~$P^{q}=\langle u^{n-r-t+1},u^{n-r-t+1,1},u^{n-r-t+2},\ldots,u^{n-r-1},u^{n-r-1,1},u^{n-r} \rangle$,

~~~~~~~~$P^{(q+1)}=\langle u^{n-r+1},u^{n-r+1,1},u^{n-r+2},\ldots,u^{n},u^{n,(n-2)^{*}},\ldots,u^{n,(n-2)^{*},\ldots,(n-k+2r-2)^{*}} \rangle$.
\\Then $V(P^i)\cap V(P^j)=\emptyset$ by Proposition~\ref{path} for $1\le i\neq j\le q+1$, and then $\mathcal{F}=\{P^{1},P^{2},\ldots,P^{q},P^{(q+1)}\}$ forms a $P_k$-structure-cut of $H_n$, one component of $H_n-\mathcal{F}$ is $\{u\}$ and $|\mathcal{F}|=\lceil\frac{2n}{k+1}\rceil$.

Thus, in either subcase, $\kappa(H_{n}; P_{k})\leq \lceil\frac{2n}{k+1}\rceil$.% and consequently $\kappa^{s}(H_{n}; P_{k})\leq \lceil\frac{2n}{k+1}\rceil$ when $k$ is odd.

\vskip .2cm
{\bf Case 2.} $k$ is even.

In this case, we denote $a:=\frac{k}{2}$. Let $n=qa+r$ for some nonnegative integers $q$ and $r$ with $0 \leq r \leq \frac{k-2}{2}$. Since $n \geq k$, we have $q\geq1$.

\vskip .2cm
{\bf Subcase 2.1.} $r=0$.

In this subcase, we set

~~~~~~~~$P^{1}=\langle u^{1},u^{1,1},u^{2},\ldots,u^{a},u^{a,1} \rangle$,

~~~~~~~~$P^{2}=\langle u^{a+1},u^{a+1,1},u^{a+2},\ldots,u^{k},u^{k,1} \rangle$,

~~~~~~~~~~~$\vdots$

~~~~~~~~$P^{q}=\langle u^{n-a+1},u^{n-a+1,1},u^{n-a+2},\ldots,u^{n},u^{n,1^{*}} \rangle$.
\\Then $V(P^i)\cap V(P^j)=\emptyset$ by Proposition~\ref{path} for $1\le i\neq j\le q$, and then $\mathcal{F}=\{P^{1},P^{2},\ldots,P^{q}\}$ forms a $P_k$-structure-cut of $H_n$, one component of $H_n-\mathcal{F}$ is $\{u\}$ and $|\mathcal{F}|=\frac{2n}{k}$.

\vskip .2cm
{\bf Subcase 2.2.} $1\leq r \leq \frac{k-2}{2}$.

In this subcase, we set

~~~~~~~~$P^{1}=\langle u^{1},u^{1,1},u^{2},\ldots,u^{a},u^{a,1} \rangle$,

~~~~~~~~$P^{2}=\langle u^{a+1},u^{a+1,1},u^{a+2},\ldots,u^{k},u^{k,1} \rangle$,

~~~~~~~~~~~$\vdots$

~~~~~~~~$P^{q}=\langle u^{n-r-a+1},u^{n-r-a+1,1},u^{n-r-a+2},\ldots,u^{n-r},u^{n-r,1} \rangle$,

~~~~~~~~$P^{(q+1)}=\langle u^{n-r+1},u^{n-r+1,1},u^{n-r+2},\ldots,u^{n},u^{n,(n-2)^{*}},\ldots,u^{n,(n-2)^{*},\ldots,(n-k+2r-2)^{*}} \rangle$.
\\Then $V(P^i)\cap V(P^j)=\emptyset$ by Proposition~\ref{path} for $1\le i\neq j\le q+1$, and then $\mathcal{F}=\{P^{1},P^{2},\ldots,P^{q},P^{(q+1)}\}$ forms a $P_k$-structure-cut of $H_n$, one component of $H_n-\mathcal{F}$ is $\{u\}$ and $|\mathcal{F}|=\lceil\frac{2n}{k}\rceil$.

Thus, $\kappa(H_{n}; P_{k})\leq \lceil\frac{2n}{k}\rceil$.% and consequently $\kappa^{s}(H_{n}; P_{k})\leq \lceil\frac{2n}{k}\rceil$ when $k$ is even.

The proof of the lemma is now complete.
\end{proof}

\begin{lemma}%Lemma 4.8
\label{Pk-lower}
Let $3 \leq k \leq n$. Then $\kappa^{s}(H_{n}; P_{k})\geq \lceil\frac{2n}{k+1}\rceil$% and $\kappa(H_{n}; P_{k})\geq \lceil\frac{2n}{k+1}\rceil$
 if $k$ is odd and $\kappa^{s}(H_{n}; P_{k})\geq \lceil\frac{2n}{k}\rceil$% and $\kappa(H_{n}; P_{k})\geq \lceil\frac{2n}{k}\rceil$
  if $k$ is even.
\end{lemma}

\begin{proof}
If no confusion should arise, we use $\mathcal{F}=\{\underbrace{P_1,\ldots,P_1}_{x_1}, \underbrace{P_2,\ldots,P_2}_{x_2}, \ldots, \underbrace{P_k,\ldots,P_k}_{x_k}\}$ to denote a set of connected subgraphs of $P_k$ with $|\mathcal{F}|=\sum^{k}_{i=1}x_i$ for $x_i\geq0$. %

If $3\le n \le 4$, then $k =3$ or $k=4$ as $3 \leq k \leq n$. So in either case, it suffices to show that $H_n-\mathcal{F}$ is connected whenever $|\mathcal{F}| \leq 1$. If $x_k=0$, then $H_n-P_i$ is connected for $1\le i\le k-1$ since $\kappa(H_n) =n\ge k$. Hence, we may assume that $x_k=1$, and consequently $x_1=\cdots=x_{k-1}=0$. That is, $\mathcal{F}=\{P_k\}$.
Assume that $H_n-P_k$ is disconnected, then each component of
$H_n-P_k$ contains at least two vertices as $H_n$ is $n$-regular and triangle-free.
On the other hand, by Lemma~\ref{kappa1}, $\kappa_1(H_n)=2n-2>k=|V(P_k)|$, a contradiction.

So, in the following, we may assume that $n \geq5$. We proceed this by contradiction.

\vskip .2cm
{\bf Case 1.} $k$ is odd.

Suppose to the contrary that $|\mathcal{F}| \leq\lceil\frac{2n}{k+1}\rceil-1$ and $H_n-\mathcal{F}$ is disconnected, then $H_n-\mathcal{F}$ has at least two components. Without loss of generality, let $C$ be the smallest component of $H_n-\mathcal{F}$. We consider the following two cases.

\vskip .2cm
{\bf Case 1.1.} $|V(C)| =1$.

In this subcase, $C$ is an isolated vertex. Let $V(C) =\{w\}$, then $|N_{H_n}(w)| =n$. By Lemma~\ref{v-p}, every element in $\mathcal{F}$ contains at most $\frac{k+1}{2}$ neighbors of $w$. Thus, $\frac{k+1}{2}|\mathcal{F}|\geq n$, i.e. $\frac{k+1}{2}(\lceil\frac{2n}{k+1}\rceil-1)\geq n$, a contradiction.

\vskip .2cm
{\bf Case 1.2.} $|V(C)| \geq2$.

In this subcase, $C$ contains at least one edge. By Lemma~\ref{kappa1}, $\kappa_1(H_n) =2n-2$. This implies that we have to delete at least $2n-2$ vertices to separate $C$ from $H_n$. However, from the assumption $|\mathcal{F}| \leq\lceil\frac{2n}{k+1}\rceil-1$, we infer that $|V(\mathcal{F})| \leq k(\lceil\frac{2n}{k+1}\rceil-1) \leq k(\frac{2n+k-1}{k+1}-1) =\frac{k}{k+1}(2n-2) <2n-2$, a contradiction.

\vskip .2cm
{\bf Case 2.} $k$ is even.

Suppose to the contrary that  $|\mathcal{F}| \leq\lceil\frac{2n}{k}\rceil-1$ with $H_n-\mathcal{F}$ being disconnected, and let $C$ be the smallest component of $H_n-\mathcal{F}$. Consider the following three cases.

\vskip .2cm
{\bf Case 2.1.} $|V(C)|=1$.

In this subcase, $C$ is an isolated vertex. Let $V(C) =\{w\}$, then $|N_{H_n}(w)| =n$. By Lemma~\ref{v-p}, every element in $\mathcal{F}$ contains at most $\frac{k}{2}$ neighbors of $w$. Thus, $\frac{k}{2}|\mathcal{F}|\geq n$, i.e. $\frac{k}{2}(\lceil\frac{2n}{k}\rceil-1)\geq n$, a contradiction.

\vskip .2cm
{\bf Case 2.2.} $|V(C)|=2$.

In this subcase, $C$ is an edge of $H_n$. Suppose that $V(C) =\{u, v\}$, then $|N_{H_n}(\{u, v\})| =2n-2$. By Lemma~\ref{kappa1}, every element in $\mathcal{F}$ contains at most $k-1$ neighbors of $\{u, v\}$. It means that we have to delete at least $\lceil\frac{2n-2}{k-1}\rceil$ vertices of $\mathcal{F}$ to separate $C$. However, $|\mathcal{F}| \leq\lceil\frac{2n}{k}\rceil-1\leq\frac{2n+k-2}{k}-1=\frac{2n-2}{k}<\lceil\frac{2n-2}{k-1}\rceil$, a contradiction.

\vskip .2cm
{\bf Case 2.3.} $|V(C)| \geq3$.

In this subcase, $|V(\mathcal{F})| \leq k(\lceil\frac{2n}{k}\rceil-1) $ from the assumption $|\mathcal{F}| \leq\lceil\frac{2n}{k}\rceil-1$. On the other hand, $\kappa_{2}(H_n)=3n-5$ by Lemma~\ref{kappa2}, which implied that we have to delete at least $3n-5$ vertices to separate $C$ from $H_n$. However, it is easily to check that $|V(\mathcal{F})| \leq k(\lceil\frac{2n}{k}\rceil-1) <3n-5$ for $n\geq5$, a contradiction.

Therefore we complete the proof of Lemma~\ref{Pk-lower}.
\end{proof}

Recall that $\kappa^{s}(H_{n};P_k)\leq \kappa(H_{n}; P_k)$, and hence, by Lemmas~\ref{Pk-upper} and~\ref{Pk-lower}, we have the following result.

\begin{theorem}
Let $3 \leq k \leq n$. Then $\kappa(H_{n}; P_{k})=\kappa^{s}(H_{n}; P_{k})=\lceil\frac{2n}{k+1}\rceil$ if $k$ is odd; $\kappa(H_{n}; P_{k})=\kappa^{s}(H_{n}; P_{k})=\lceil\frac{2n}{k}\rceil$ if $k$ is even.
\end{theorem}

\section{Conclusions}

In this paper, we consider the $T$-structure connectivity and $T$-substructure connectivity of the twisted hypercube $H_n$ for $T\in\{K_{1,r}, P_k\}$ and show that
\vskip.2cm

~~~~~~~~~~~~~$\kappa(H_{n}; K_{1,r})=\kappa^{s}(H_{n}; K_{1,r})=\lceil \frac{n}{2}\rceil$ for $n\geq 4$ and $3\leq r\leq 4$;

$$\kappa(H_{n}; P_k)=\kappa^{s}(H_{n}; P_k)=\left\{
\begin{array}{ll}
n,& \mbox{if $k=1$,}\\
n-1,  & \mbox{if $k=2$ and $\kappa(n-1)=1$,}\\
n,  & \mbox{if $k=2$ and $\kappa(n-1)\geq 2$,}\\
\lceil\frac{2n}{k+1}\rceil,  & \mbox{if $3\leq k\leq n$ and $k$ is odd,}\\
\lceil\frac{2n}{k}\rceil,  & \mbox{if $4\leq k\leq n$ and $k$ is even.}\\
\end{array}
\right.$$

However, determining the $K_{1,r}$-structure connectivity and $K_{1,r}$-substructure connectivity of $H_n$ with  $r\ge 5$ remain open.
To explore $\kappa(H_n; C_k)$ and $\kappa^{s}(H_n; C_k)$, the approach used in this paper is invalid as $H_n$ is a nonbipartite graph.
But one may explore the structure connectivity and substructure connectivity of other interconnection networks by the approach used in this paper.

\vskip 0.6cm

\noindent{\bf \Large Acknowledgments}

\vskip 0.3cm

\noindent Xiaolan Hu is partially supported
by NNSFC under grant number 11601176 and  NSF of Hubei Province under grant number 2016CFB146. Huiqing Liu is partially supported by NNSFC under grant numbers 11571096 and 61373019.


\vskip 0.2cm

\noindent

\begin{thebibliography}{99}

\bibitem{Bondy08}  J.A. Bondy, U.S.R. Murty, Graph Theory, Springer, 2008.

\bibitem{Chang13}  N.W. Chang, S.Y. Hsieh, $\{2, 3\}$-extraconnectivities of hypercube-like networks, J. Comput. System Sci. 79(5)(2013) 669-688.

\bibitem{Chen03} Y.C. Chen, J.J.M. Tan, L.H. Hsu, Super connectivity and super-edge-connectivity for some interconnection networks, Appl. Math. Comput.
140(2)(2003) 245-254.

\bibitem{Esfahanian} A.H. Esfahanian, Generalized measures of fault tolerance with application to $n$-cube networks, IEEE Trans. Comput. 38(11)(1989) 1586-1591.

\bibitem{F} J. F\'{a}brega, M.A. Fiol, On the extraconnectivity of graphs, Discrete Math. 155(1-3)(1996) 49-57.

\bibitem{Fan03} J.X. Fan, L.Q. He, BC interconnection networks and their properties, Chinese J. Comput. 26(1)(2003) 1-7.

\bibitem{Harary} F. Harary, J.P. Hayes, H.J. Wu, A survey of the theory of hypercube graphs, Comput. Math. Appl. 15(4)(1988) 277-289.

\bibitem{Hsieh12} S.Y. Hsieh, Y.H. Chang, Extraconnectivity of $k$-ary $n$-cube networks, Theoret. Comput. Sci. 443(20)(2012) 63-69.

\bibitem{Lin16} C.K. Lin, L.L. Zhang, J.X. Fan, D.J. Wang, Structure connectivity and substructure connectivity of hypercubes, Theoret. Comput. Sci. 634(2016) 97-107.

\bibitem{Liu} H.Q. Liu, X.L. Hu, S. Gao, The $g$-good neighbor conditional diagnosability of twisted hypercubes under the PMC and MM$^*$ model,
Appl. Math. Comp. to appear.
%
%\bibitem{Lv17} Y.L. Lv, C.K. Lin, J.X. Fan, Hamiltonian cycle and path embeddings in $k$-ary $n$-cubes based on structure faults, The Computer Journal 60(2) (2017) 159-179.

\bibitem{Lv18} Y.L. Lv, J.X. Fan, D.F. Hsu, C.K. Lin, Structure connectivity and substructure connectivity of $k$-ary $n$-cube networks, Inform. Sci.
433-434(2018)  115-124.

%\bibitem{LvXu} Y.L. Lv, Y.L. Xu, $K_{1,1}$-structure connectivity and $K_{1,1}$-substructure connectivity of torus networks,

\bibitem{Mane} S.A. Mane, Structure connectivity of hypercubes, AKCE International Journal of Graphs and Combinatorics, https://doi.org/10.1016/j.akcej.2018.01.009.


\bibitem{Qi} H. Qi, X.D. Zhu, The fault-diameter and wide-diameter of twisted hypercubes, Discrete Appl. Math. 235(2018) 154-160.

\bibitem{Sabir17} E. Sabir, J. Meng, Structure fault tolerance of hypercubes and folded hypercubes, Theoret. Comput. Sci. 711(2018) 44-55.

\bibitem{Xu10} J.M. Xu, J.W. Wang, W.W. Wang, On super and restricted connectivity of some interconnection networks, Ars Combin. 94(2010).

%\bibitem{Xu05} J.M. Xu, Q. Zhu, X.M. Hou, T. Zhou, On restricted connectivity and extra connectivity of hypercubes and folded hypercubes, J. Shanghai Jiaotong Univ. 10(2)(2005) 203-207.

\bibitem{Xu07} J.M. Xu, Q. Zhu, M. Xu, Fault-tolerant analysis of a class of networks, Inform. Process. Lett. 103(6)(2007) 222-226.

\bibitem{Yang09} W.H. Yang, J.X. Meng, Extraconnectivity of hypercubes, Appl. Math. Lett. 22(6)(2009) 887-891.

%\bibitem{Zhu08} Q. Zhu, On conditional diagnosability and reliability of the BC networks, J. Supercomput. 45(2) (2008) 173-184.

\bibitem{Zhu07} Q. Zhu, J.M. Xu, X.M. Hou, M. Xu, On reliability of the folded hypercubes, Inform. Sci. 177(8)(2007) 1782-1788.

\bibitem{Zhu17}  X. Zhu, A hypercube variant with small diameter, Journal of Graph Theory 85(3)(2017) 651-660.

\end{thebibliography}
\end{document}